\documentclass{amsart}
\usepackage{amsxtra,amscd}
\usepackage{graphicx}
\usepackage{amsmath}
\usepackage{amsfonts}
\usepackage{amssymb}

\newtheorem{theorem}{Theorem}[section]
\newtheorem{lemma}[theorem]{Lemma}
\theoremstyle{definition}
\newtheorem{definition}[theorem]{Definition}

\newtheorem{remark}[theorem]{Remark}

\numberwithin{equation}{section}

\begin{document}
\title{On the arithmetic Galois covers of higher relative dimensions}
\author{Feng-Wen An}
\address{School of Mathematics and Statistics, Wuhan University, Wuhan,
Hubei 430072, People's Republic of China}
\email{fwan@amss.ac.cn}

\begin{abstract}
In this paper we will give the calculus, the criterion, and the existence of the arithmetic Galois covers of higher relative dimensions.
\end{abstract}

\maketitle

\begin{center}
{\tiny {Contents} }
\end{center}

{\tiny \qquad {Introduction} }

{\tiny \qquad {1. Review of Transcendental Galois Extensions} }

{\tiny \qquad {2. Galois Covers of Higher Relative Dimensions} }

{\tiny \qquad {3. Proof of Theorem 2.3} }

{\tiny \qquad {4. Proof of Theorem 2.6}}

{\tiny \qquad {References}}

\section*{Introduction}

Let $L$ be an extension of a field $K$ (not necessarily algebraic). We can define the Galois extension for $L/K$ in a manner similar to algebraic extensions of fields, i.e., $L$ is  Galois over $K$ if $K$ is the invariant subfield of the Galois group $Gal(L/K)$. However, the situation of a transcendental extension $L/K$ is more complicated. The Galois group $Gal(L/K)$ is always an infinite group; many approaches established upon finite dimensions  will not be valid in general (see \cite{1004}).

In fact, there exist several types of (general) Galois extensions of fields, from the weakest one to the strongest one, such as: \emph{Galois}; \emph{tame Galois}; \emph{strong Galois}; \emph{absolute Galois} (see \S 1 for definitions). These general Galois extensions coincide with each other for algebraic extensions but in general are distinct from each other for transcendental ones (for instance, see \emph{Remark 1.3}; for counterexamples, see \cite{1004}).

Given two arithmetic varieties $X/Y$, i.e., two integral schemes surjectively over $Spec\mathbb{Z}$ of finite dimensions. The geometry of the arithmetic varieties $X/Y$ encodes and determines a good quantity of information of the function fields $k(X)/k(Y)$. In deed, there exists a nice relationship between the geometry of the arithmetic varieties and the arithmetic of the function fields. If $\dim X=\dim Y$, in particular, $k(X)$ is a finite extension of $k(Y)$, there are many well-known results on the topic or related (for instance, see \cite{sga1,s-v1,s-v2}).

Now consider a general case that $\dim X \geqq \dim Y$, i.e., the function field $k(X)$ is a transcendental extension of $k(Y)$. Then the situation will be changed considerably without mentioning the transcendental Galois extensions of the function fields. This leads us discuss the properties between arithmetic Galois covers of higher relative dimensions and Galois groups of the function fields.

In fact, the arithmetic Galois covers of higher relative dimensions are encountered by us
in the discussions on arithmetic function fields enlightened to some extent by Lang's unramified class field theory of function fields in several variables (see \cite{lang}).

We say that the arithmetic varieties $X/Y$ are a \emph{geometric model} of the fields $L/K$ if for the function fields we have $k(X)=L$ and $k(Y)=K$. In such a case, the arithmetic varieties $X/Y$  will give us a geometric vivid imagination of the arithmetic function fields $L/K$, in particular, the automorphism group $Aut(X/Y)$ will be naturally isomorphic to the Galois group $Gal(L/K)$ of the fields; conversely, as fields are more accessible, we can change the fields $L/K$ to control the arithmetic varieties $X/Y$ in a birational geometry.

Hence, first of all, there are several problems needing to be dealt with:
\begin{itemize}
\item Let $X$ and $Y$ be two arithmetic varieties such that $\dim X \geqq \dim Y$. When does $Aut(X/Y)\cong Gal(k(X)/k(Y))$  hold naturally as groups?

\item Let $X$ be an arithmetic variety. Given any extension $L$ of the function field $K\triangleq k(X)$. Does there exist any arithmetic variety $X_{L}$ such that $X_{L}/X$ are a geometric model of the fields $L/K$? How to construct $X_{L}$ when the field $L$ varies like a function?

\item Or return to the theory of class fields. Let $X$ be an arithmetic variety. How to use the geometric data of $X$ to describe the class fields over the function field $K\triangleq k(X)$ (in several variables)?
\end{itemize}

In this paper we will have several discusses on the topics above. For any two arithmetic varieties $X/Y$ such that $\dim X \geqq \dim Y$, we will introduce the \emph{local completeness} (\emph{Definition 2.2}) to give a criterion for the identity that $Aut(X/Y)\cong Gal(k(X)/k(Y))$  holds naturally (\emph{Theorem 2.3}). For the case of normal schemes, the condition of local completeness is automatically satisfied.

Given an arbitrary extension $L$ of the function field $k(X)$. We will construct an arithmetic variety $X_{L}\triangleq X[\Delta _{L/k(X)}]$ such that $k(X_{L})=L$ and that $X_{L}$ is affine and surjective over the arithmetic variety $X$; moreover, if $L$ is Galois over $k(X)$, $X_{L}$ is correspondingly Galois over $X$ in the field $L$ (\emph{Theorem 2.6}). Here, $\Delta _{L/k(X)}$ denotes a set of generators of $L$ over $k(X)$. When $X$ and $L$ vary, we have the \emph{calculus of arithmetic Galois covers} $X[\Delta _{L/k(X)}]$ over $X$ in $L$ (see \S 2.5).

As an application, by these results above we can give a  computation of the \'{e}tale fundamental group of an arithmetic scheme. For an arithmetic variety $X$, there is a natural isomorphism $\pi^{et}_{1}(X)\cong Gal(k(X)^{au}/k(X))$ between groups, that is, the \'{e}tale fundamental group $\pi^{et}_{1}(X)$ of $X$ is  isomorphic to the Galois group $Gal(k(X)^{au}/k(X))$ of the maximal arithmetically unramified extension $k(X)^{au}$ of the function field $k(X)$ in a natural manner (see \cite{0910}).
Here, the arithmetically unramified, defined as one has done in algebraic number theory (for instance, see \cite{neu}),  coincides with that in algebraic number theory (see \cite{1006}).

Hence, the \'{e}tale fundamental group $\pi^{et}_{1}(X)$ of an arithmetic variety $X$ encodes all unramified extensions of the function field $k(X)$ (in several variables).

\section{Review of Transcendental Galois Extensions}

Let $L$ be an extension of a field $K$ (not necessarily algebraic).
In this section we will review several types of Galois extensions of fields: \emph{Galois}; \emph{tame Galois}; \emph{strong Galois}; \emph{absolute Galois} (see \cite{1004}). Those Galois extensions of fields, passing
from the weakest one to the strongest one,  coincide with each other for algebraic extensions but are different from each other for transcendental extensions in general.

\subsection{Recalling several types of transcendental Galois extensions of fields}

As usual, the \textbf{Galois group} of $L$ over $K$, namely $Gal(L/K)$, is the group of all automorphisms $\sigma$ of $L$ such that $\sigma (x)=x$ holds for any $x\in K$. Denote by $tr.deg L/K$ the transcendental degree of $L$ over $K$.

By a \textbf{nice basis} $(\Delta, A)$ of $L$ over $K$, we understand that $\Delta$ is a transcendental basis of $L$ over $K$ and $A$ is a linearly basis of $L$ over $K(\Delta)$ such that $L=K(\Delta)[A]$ is algebraic over $K(\Delta)$.

Furthermore, such a nice basis $(\Delta, A)$ is  a \textbf{linearly disjoint basis} of $L$ over $K$ if the fields $K(\Delta)$ and $K(A)$ are linearly disjoint over $K$.
Put
\begin{itemize}
\item $\Sigma[L/K]\triangleq$ the set of nice bases of $L$ over $K$;

\item $\Sigma[L/K]_{ld}\triangleq$ the set of linearly disjoint basis of $L$ over $K$.
\end{itemize}

From \emph{Definition 1.1} below it is seen that a transcendental Galois extension is more complicated than an algebraic one.

\begin{definition}
(\cite{1004})
There exist several types of Galois extensions of fields such as the following.
\begin{itemize}

\item $L$ is said to be \textbf{Galois} over $K$  if $K$ is the
invariant subfield of $Gal(L/K)$, i.e., if $K=\{x\in L :\sigma (x)=x
\text{ for any }\sigma \in Gal(L/K) \}.$

\item  $L$ is said to be \textbf{tame Galois} over $K$  if there is some $(\Delta, A)\in \Sigma[L/K]$ such that $L$ is Galois over the subfield $K(\Delta)$.

\item  $L$ is said to be \textbf{strong Galois} over $K$  if the two conditions are satisfied:
     \begin{itemize}
     \item $\Sigma[L/K]_{ld}$ is a nonempty set;

     \item $L$ is an algebraic Galois extension of $K(\Delta)$ for any $(\Delta, A)\in\Sigma[L/K]_{ld}$.
     \end{itemize}

\item  $L$ is said to be \textbf{absolute Galois} over $K$  if $L$ is Galois over any subfield $F$ such that $K\subseteq F\subseteq L$.
\end{itemize}
\end{definition}

\begin{remark}
(\cite{1004})
For the Galois extensions, we have
\begin{equation*}
\begin{array}{l}
\text{absolute Galois} \\

\implies \text{Galois};
\end{array}
\end{equation*}
\begin{equation*}
\begin{array}{l}
\text{absolute Galois (with linear disjoint bases)} \\

\implies  \text{strong Galois} \\

 \implies  \text{tame Galois} \\

  \implies \text{Galois}.
\end{array}
\end{equation*}
\end{remark}

\begin{remark}
(\cite{1004})
There is a comparison between algebraic Galois extensions and transcendental Galois extensions.
In a manner similar to \emph{conjugate} and \emph{normal} in algebraic extensions of fields, we have notions of \emph{conjugations} and \emph{quasi-galois} for transcendental extensions of fields. For algebraic extensions,  we have
$$\emph{Galois} =\emph{a unique conjugation} + \emph{separably generated}.$$
Likewise, for transcendental extensions of fields, we have
$$\emph{strong Galois} =\emph{a unique strong conjugation} + \emph{separably generated};$$
$$\emph{absolute Galois} =\emph{a unique absolute conjugation} + \emph{separably generated}.$$
However,
$$\emph{tame Galois} \not=\emph{a unique conjugation} + \emph{separably generated}.$$
See \cite{1004} for examples and counterexamples of the transcendental Galois extensions.
\end{remark}

\section{Galois Covers of Higher Relative Dimensions}

In this section we will discuss the calculus, the criterion, and the existence of the arithmetic Galois covers of higher relative dimensions.

\subsection{Conventions}

For brevity, by an \textbf{arithmetic variety} we will understand a finite-dimensional integral scheme that is surjectively over $\mathbb{Z}$ (not necessarily of finite type).

For an arithmetic variety $Z$, let
$k(Z)\triangleq \mathcal{O}_{X,\xi}$, i.e., the function field of $Z$, where $\xi$ is the generic point of $Z$.

\subsection{Definitions for Galois covers of higher relative dimensions}

Let $X$ and $Y$ be arithmetic varieties and let $f:X\rightarrow Y$ be a surjective morphism. Note that here the relative dimension $\dim X - \dim Y$ can be permitted to be positive. Suppose that
 $Aut\left( X/Y\right) $ is the group of automorphisms of $X$ over $Y$ by $f$.

We have several types of higher relative dimensional Galois covers over arithmetic varieties.

\begin{definition}
The arithmetic variety $X$ is \textbf{Galois} (or \textbf{tame Galois}, \textbf{strong Galois}, \textbf{absolute Galois}, respectively)
over $Y$ by $f$  if the two conditions are satisfied:
\begin{itemize}
\item $Aut(X/Y)\cong Gal(k(X)/k(Y))$ are canonically isomorphic groups;

\item $k(X)$ is naturally Galois (or tame Galois, strong Galois, absolute Galois, respectively) over $k(Y)$.
\end{itemize}
\end{definition}

\begin{definition}
The arithmetic variety $X$ is \textbf{locally complete} over $Y$ by $f$ if there is an affine open set $W\subseteq f^{-1}(V)$ and an isomorphism $\widetilde{\rho}:\mathcal{O}_{X}(U)\to \mathcal{O}_{X}(W)$ of algebras over $\mathcal{O}_{Y}(V)$ such that $\rho$ is naturally induced from $\widetilde{\rho}$ for each automorphism $\rho\in Gal(k(X)/f^{\sharp}(k(Y)))$ and each affine open sets $V\subseteq Y$ and $U\subseteq f^{-1}(V)$.
\end{definition}

\subsection{Criterion for Galois covers of higher relative dimensions}

Here there is a criterion for an arithmetic Galois cover of a higher relative dimension.

\begin{theorem}
Let $X$ and $Y$ be two arithmetic varieties such that $\dim X\geqq\dim Y$. Suppose that there is a surjective morphism $\phi:X\rightarrow Y$ such that
\begin{itemize}
\item $X$ is locally complete over $Y$;

\item $k(X)$ is Galois \emph{(}or tame Galois, strong Galois, absolute Galois, respectively\emph{)} over $k(Y)$ by $f$.
\end{itemize}

Then $X$ is Galois \emph{(}or tame Galois, strong Galois, absolute Galois, respectively\emph{)} over $Y$.
\end{theorem}

We will prove \emph{Theorem 2.3} in \S 3.

\begin{remark}
The conclusion of \emph{Theorem 2.3} can be regarded as a generalization of several well-known related results in \cite{sga1,s-v1,s-v2} for the case that  $Aut(X/Y)$ are finite groups; at the same time, it is a generalization of the main result in \cite{1001} for the case that $\dim X=\dim Y$ and $Aut(X/Y)$ is an infinite group.
\end{remark}

\subsection{Existence for Galois covers of higher relative dimensions}

Let $X$ be an arithmetic variety and let $L$ be an extension of the function field $k(X)$. Note that here $L$ has a nonnegative transcendental degree over $k(X)$, i.e., $tr.deg L/K \geqq 0$.

\begin{definition}
Suppose that $L$ is Galois \emph{(}or tame Galois, strong Galois, absolute Galois, respectively\emph{)} over $k(X)$. An arithmetic variety $X_{L}$  is said to be a \textbf{Galois} (or \textbf{tame Galois}, \textbf{strong Galois}, \textbf{absolute Galois}, respectively) \textbf{cover} of $Y$ \textbf{in a field} $L$ if there is a surjective morphism $\phi_{L}:X_{L}\to X$ satisfying the following conditions:
\begin{itemize}
\item $k(X_{L})=L$;

\item $\phi_{L}$ is affine;

\item $X_{L}$ is Galois (or tame Galois, strong Galois, absolute Galois, respectively) over $X$ by $\phi_{L}$;

\item For any affine open set $V\subseteq X$, the ring
$\mathcal{O}_{X}(V)$ is isomorphic to the invariant subring of $\mathcal{O}_{X_{L}}(\phi_{L}^{-1}(V))$ under the natural action of the group ${Aut(X_{L}/X)}$.
\end{itemize}
\end{definition}

We have the following existence of an arithmetic Galois cover with a higher relative dimension for any prescribed Galois extension of the function field.

\begin{theorem}
Let $X$ be an arithmetic variety and let $L$ be a Galois \emph{(}or tame Galois, strong Galois, absolute Galois, respectively\emph{)} extension of the function field $k(X)$ such that $tr.deg L/K \geqq 0$.

Then there is an arithmetic variety $X_{L}$ and a  morphism $\phi_{L}:X_{L}\to X$ such that
$X_{L}$ is Galois \emph{(}or tame Galois, strong Galois, absolute Galois, respectively\emph{)} over $X$ in the field $L$.
\end{theorem}

We will prove \emph{Theorem 2.6} in \S 4.
Noticed that for a fixed field $L$, in general, $X$ can have many Galois covers $X_{L}$ in $L$ which are not isomorphic with each other (see \emph{Remarks 2.-10} below).

\begin{remark}
The conclusion of \emph{Theorem 2.6} is a generalization of the main result in \cite{0909}.
Such a result is one of the key points for us to give a computation of the \'{e}tale fundamental group of an integral scheme (see \cite{0910}).
\end{remark}

\subsection{Calculus for Galois covers of higher relative dimensions}

Now consider the calculus of Galois covers of arithmetic varieties.

\begin{definition}
(c.f. \cite{0910})
Let $Y$ be an arithmetic variety. Suppose that $L$ is a Galois (or tame Galois, strong Galois, absolute Galois, respectively) extension of the function field $k(Y)$.
We have the following \textbf{calculus of arithmetic Galois covers}
\begin{itemize}
\item $Y[\Delta]\triangleq$  a Galois (or tame Galois, strong Galois, absolute Galois, respectively)  cover $X$ of $Y$ in $L$, obtained by adding a set $\Delta\subseteq L \setminus k(Y)$ of generators of $L$ over $k(Y)$  to an essentially affine realization of $Y$ in the
 construction of Galois covers in \S 4.3 below.
\end{itemize}
\end{definition}

\begin{remark}
(c.f. \cite{0910})
Fixed an arithmetic variety $Y$ and a Galois (or tame Galois, strong Galois, absolute Galois, respectively) extension $L$ of $k(Y)$. Let $\Sigma^{\natural}[L/Y]$ be the family of sets $\Delta\subseteq L \setminus k(Y)$ of generators of $L$ over $k(Y)$.
There are the following statements.
\begin{itemize}
\item Different sets $\Delta\in \Sigma^{\natural}[L/Y]$  can produce different Galois (or tame Galois, strong Galois, absolute Galois, respectively) covers $Y[\Delta]$ of $Y$ in $L$. In general, it is not true that $Y[\Delta]$ and $Y[\Delta^{\prime}]$ are isomorphic over $Y$ for two different $\Delta,\Delta^{\prime}\in \Sigma^{\natural}[L/Y]$.

\item There is a biggest Galois (or tame Galois, strong Galois, absolute Galois, respectively) cover $Y[\Delta_{min}]$ of $Y$ in $L$,  which contains all possible points that can be added to the underlying space.

\item There is a smallest Galois (or tame Galois, strong Galois, absolute Galois, respectively) cover $Y[\Delta_{max}]$ of $Y$ in $L$ such that the underlying space is the smallest and can be regarded as that of the scheme $Y$.
\end{itemize}
\end{remark}

\begin{remark}
(c.f. \cite{0910})
Let $Y$ be an arithmetic variety and let $L$ be a  Galois (or tame Galois, strong Galois, absolute Galois, respectively) extension of $k(Y)$.
Then
\begin{equation*}
Y[\Delta\cup {\Delta}^{-1}]\cong Y[\Delta^{\prime}\cup {\Delta^{\prime}}^{-1}]
\end{equation*}
are isomorphic schemes over $Y$ for any $\Delta,\Delta^{\prime}\in \Sigma^{\natural}[L/Y]$. In particular, each $Y[\Delta\cup {\Delta}^{-1}]$ is a smallest Galois (or tame Galois, strong Galois, absolute Galois, respectively) cover of $Y$ in $L$.
 Here, ${\Delta}^{-1}=\{x^{-1}\in L:x\in \Delta\}$.
\end{remark}

\section{Proof of Theorem 2.3}

In this section we will have a new sufficient condition for geometric models of fields (see  \emph{Lemma 3.1}) (c.f. \cite{0907,1001,1004}). Note that here the approach is not essentially new. Then from the lemma a proof of
 \emph{Theorem 2.3} will be obtained immediately.

\subsection{A preparatory lemma}

Let $X$ be an arithmetic variety with generic point $\xi$. Recall that for each open set $U$ in $X$, there is a natural inclusion $$i_{X}:\mathcal{O}_{X}(U)\to k(X)\triangleq \mathcal{O}_{X,\xi}$$ between rings. We have $$k(X)=\{i_{X}(x)\in k(X):x\in \mathcal{O}_{X}(U), U \text{ is open in }X\}.$$
At the same time, $k(X)$ will be taken as the rational field $$\{[U,f]:f\in \mathcal{O}_{X}(U), U \text{ is open in }X\}$$ where $[U,f]$ denotes the germ of $(U,f)$ (see \cite{ega,gtm52,iitaka} for detail).

Now let's give a sufficient condition that for two arithmetic varieties $X/Y$, the automorphism group $Aut(X/Y)$ is naturally isomorphic to the Galois group $Gal(k(X)/k(Y))$ of the function field $k(X)$ over $k(Y)$.

\begin{lemma}
Let $X$ and $Y$ be arithmetic varieties with $\dim Y\leqq\dim X $. Suppose that $X$ is locally complete over $Y$ by a surjective morphism $\phi:X\rightarrow Y$. Then there is a natural
isomorphism
\begin{equation*}
Aut\left( X/Y\right) \cong Gal\left( k\left( X\right) /k\left( Y\right)
\right)
\end{equation*}
between groups.
\end{lemma}

\begin{proof}
(The approach here is based on a
trick originally in \cite{0907} and a refinement in \cite{0910,1001}.)
In the following we will proceed in several steps
to prove that there is  a group isomorphism
\begin{equation*}
t:Aut\left( X/Y\right) \longrightarrow Gal\left( k\left( X\right) /k\left(
Y\right) \right)\triangleq Gal\left( k\left( X\right)
/\phi^{\sharp}\left(k\left( Y\right)\right) \right)
\end{equation*}
by
\begin{equation*}
\sigma =(\sigma ,\sigma ^{\sharp })\longmapsto t(\sigma )=\left\langle
\sigma ,\sigma ^{\sharp -1}\right\rangle
\end{equation*}
where $\left\langle \sigma ,\sigma ^{\sharp -1}\right\rangle $ is the map of
$k(X)$ into $k(X)$ given by
\begin{equation*}
[ U,f] \in k\left( X\right)
\longmapsto [ \sigma \left( U\right) ,\sigma ^{\sharp -1}\left(
f\right) ] \in  k\left( X\right)
\end{equation*}
for any open set $U$ in $X$ and any element $f\in \mathcal{O}_{X}(U)$.

Here
the function field $k(X)$ is taken canonically as the set of equivalence classes $[U,f]$ of elements of the
form ${\left( U,f\right)}$ such that $U$ is an open set in $X$ and $f$ is contained in $\mathcal{O}_{X}(U)$.

\textbf{Step 1.} Construct a map $$t:Aut\left( X/Y\right) \longrightarrow Gal\left( k\left( X\right) /\phi^{\sharp}(k\left(
Y\right) )\right).$$

In fact, fixed any automorphism $$ \sigma =\left( \sigma
,\sigma ^{\sharp}\right) \in Aut_{k}\left( X/Y\right) . $$ That is to say, $$
\sigma : X \longrightarrow X $$ is a homeomorphism; $$ \sigma ^
{\sharp}:\mathcal{O}_{X} \rightarrow \sigma _{\ast }\mathcal{O}_{X}
$$ is an isomorphism of sheaves of rings on $X$.

As $\dim X<\infty$,
we have $\sigma (\xi)=\xi$. Then
\begin{equation*}
\sigma ^{\sharp}:k\left( X\right)=\mathcal{O}_{X,\xi } \rightarrow \sigma
_{\ast }\mathcal{O}_{X,\xi }=k\left( X\right)
\end{equation*}
is an automorphism of $k(X)$.

Denote by $\sigma ^{\sharp-1}$ the inverse of the ring isomorphism $$\sigma ^{\sharp}:k(X)\to k(X).$$

Take any open subset $U$ of $X$. We have the restriction
\begin{equation*}
\sigma=(\sigma ,\sigma ^{\sharp}): (U,\mathcal{O}_{X}|_{U}) \longrightarrow
(\sigma(U),\mathcal{O}_{X}|_{\sigma(U)})
\end{equation*}
of open subschemes. That is,
\begin{equation*}
\sigma^{\sharp}:\mathcal{O}_{X}|_{\sigma(U)} \rightarrow \sigma_{\ast}%
\mathcal{O}_{X}|_{U}
\end{equation*}
is an isomorphism of sheaves on $\sigma (U)$. In particular,
\begin{equation*}
\sigma^{\sharp}:\mathcal{O}_{X}(\sigma(U))
=\mathcal{O}_{X}|_{\sigma(U)}(\sigma(U)) \rightarrow \mathcal{O}_{X}(U)=\sigma_{\ast}
\mathcal{O}_{X}|_{U}(\sigma(U))
\end{equation*}
is an isomorphism of rings.

For every $f \in \mathcal{O}_{X}|_{U}(U)$, we have
\begin{equation*}
f \in \sigma_{\ast}\mathcal{O}_{X}|_{U}(\sigma(U));
\end{equation*}
hence
\begin{equation*}
\sigma^{\sharp -1}(f) \in \mathcal{O}_{X}(\sigma(U)).
\end{equation*}

Now define a set mapping
\begin{equation*}
t:Aut_{k}\left( X/Y\right) \longrightarrow Gal\left( k\left( X\right)
/\phi^{\sharp}(k\left( Y\right) )\right)
\end{equation*}
given by
\begin{equation*}
\sigma =(\sigma ,\sigma ^{\sharp})\longmapsto t(\sigma)=\left\langle \sigma ,\sigma ^{\sharp
-1}\right\rangle
\end{equation*}
such that
\begin{equation*}
\left\langle \sigma ,\sigma ^{\sharp-1}\right\rangle :[ U,f]
\longmapsto [ \sigma \left( U\right) ,\sigma ^{\sharp-1}\left( f\right)
]
\end{equation*}
is a mapping of $k(X)$ into $k(X)$.

\textbf{Step 2.} Prove that $$t:Aut\left( X/Y\right) \longrightarrow Gal\left( k\left( X\right) /\phi^{\sharp}(k\left(
Y\right)) \right)$$ is a group homomorphism.

In deed, given any
\begin{equation*}
\sigma =\left( \sigma ,\sigma ^{\sharp}\right) \in Aut\left( X/Y\right).
\end{equation*}
For any $[U,f],[V,g] \in k(X)$, we have
\begin{equation*}
[U,f]+[V,g]=[U\cap V, f+g]
\end{equation*}
and
\begin{equation*}
[U,f]\cdot [V,g]=[U\cap V, f\cdot g];
\end{equation*}
then we have
\begin{equation*}
\begin{array}{l}
\left\langle \sigma ,\sigma ^{\sharp-1}\right\rangle([U,f]+[V,g]) \\
=\left\langle \sigma ,\sigma ^{\sharp-1}\right\rangle([U\cap V, f+g]) \\
=[\sigma(U\cap V), \sigma^{\sharp -1}(f+g)] \\
=[\sigma(U\cap V), \sigma^{\sharp -1}(f)]+[\sigma(U\cap V), \sigma^{\sharp
-1}(g)] \\
=[\sigma(U), \sigma^{\sharp -1}(f)]+[\sigma(V), \sigma^{\sharp -1}(g)]\\
=\left\langle \sigma ,\sigma ^{\sharp-1}\right\rangle([U,f])+ \left\langle
\sigma ,\sigma ^{\sharp-1}\right\rangle([V,g])
\end{array}
\end{equation*}
and
\begin{equation*}
\begin{array}{l}
\left\langle \sigma ,\sigma ^{\sharp-1}\right\rangle([U,f]\cdot[V,g]) \\
=\left\langle \sigma ,\sigma ^{\sharp-1}\right\rangle([U\cap V, f\cdot g])
\\
=[\sigma(U\cap V), \sigma^{\sharp -1}(f\cdot g)] \\
=[\sigma(U\cap V), \sigma^{\sharp -1}(f)]\cdot[\sigma(U\cap V),
\sigma^{\sharp -1}(g)] \\
=[\sigma(U), \sigma^{\sharp -1}(f)]\cdot[\sigma(V), \sigma^{\sharp -1}(g)]
\\
=\left\langle \sigma ,\sigma ^{\sharp-1}\right\rangle([U,f])\cdot
\left\langle \sigma ,\sigma ^{\sharp-1}\right\rangle([V,g]).
\end{array}
\end{equation*}

It follows that $\left\langle \sigma ,\sigma ^{\sharp-1}\right\rangle$ is an
automorphism of $k\left( X\right) .$

It needs to prove that $\left\langle \sigma ,\sigma
^{\sharp-1}\right\rangle$ is an isomorphism over
$\phi^{\sharp}(k(Y))$. Consider the given morphism
$$\phi=(\phi,\phi^{\sharp}):(X,\mathcal{O} _{X})\rightarrow
(Y,\mathcal{O}_{Y})$$ of schemes. It follows that $\phi(\xi)$ is the
generic point of $Y$ and $\xi$ is invariant under any automorphism
$\sigma \in Aut\left( X/Y\right)$. Then $\sigma^{\sharp}:
\mathcal{O}_{X,\xi} \rightarrow \mathcal{O}_{X,\xi}$ is an
isomorphism of algebras over
$\phi^{\sharp}(\mathcal{O}_{Y,\phi(\xi)})=\phi^{\sharp}(k(Y))$.
Hence, $$\left\langle \sigma ,\sigma
^{\sharp-1}\right\rangle|_{\phi^{\sharp}(k(Y))}=id_{\phi^{\sharp}(k(Y))}.$$

This proves
$$
\left\langle \sigma ,\sigma ^{\sharp-1}\right\rangle \in Gal\left( k\left(
X\right) /\phi^{\sharp}(k\left( Y\right) \right)) .
$$ So, as a map of sets, $t$ is  well-defined.

Prove that $t$ is a homomorphism between groups. In fact, take any
\begin{equation*}
\sigma =\left( \sigma ,\sigma ^{\sharp}\right) ,\delta =\left( \delta
,\delta ^{\sharp}\right) \in Aut\left( X/Y\right) .
\end{equation*}
By preliminary facts on schemes (see \cite{ega,gtm52,iitaka}) we have
\begin{equation*}
\delta ^{\sharp -1}\circ \sigma^{\sharp -1}=(\delta \circ \sigma)^{\sharp -1};
\end{equation*}
then
\begin{equation*}
\left\langle \delta ,\delta ^{\sharp-1}\right\rangle \circ \left\langle
\sigma ,\sigma ^{\sharp-1}\right\rangle =\left\langle \delta \circ \sigma
,\delta ^{\sharp-1}\circ \sigma ^{\sharp-1}\right\rangle.
\end{equation*}

Hence, the map
\begin{equation*}
t:Aut\left( X/Y\right) \rightarrow Gal\left( k\left( X\right)
/\phi^{\sharp}\left(k\left( Y\right)\right) \right)
\end{equation*}
is a homomorphism of groups.

\textbf{Step 3.} Prove that $$t:Aut\left( X/Y\right) \longrightarrow Gal\left( k\left( X\right) /\phi^{\sharp}(k\left(
Y\right)) \right)$$ is an injective homomorphism.

In deed, assume $\sigma ,\sigma
^{\prime }\in {Aut}\left( X/Y\right) $ such that $t\left( \sigma
\right) =t\left( \sigma ^{\prime }\right) .$ We have
\begin{equation*}
[ \sigma \left( U\right) ,\sigma ^{\sharp-1}\left( f\right)]
=[ \sigma ^{\prime }\left( U\right) ,\sigma ^{\prime \sharp-1}\left(
f\right) ]
\end{equation*}
for any $[ U,f] \in k\left( X\right) .$ In particular, we
have
\begin{equation*}
\left( \sigma \left( U_{0}\right) ,\sigma ^{\sharp-1}\left( f\right) \right)
=\left( \sigma ^{\prime }\left( U_{0}\right) ,\sigma ^{\prime
\sharp-1}\left( f\right) \right)
\end{equation*}
for any $f\in \mathcal{O}_{X}(U_{0})$ and any affine open subset
$U_{0}$ of $X$ such that $ \sigma \left( U_{0}\right)$ and $\sigma
^{\prime }\left( U_{0}\right)$ are both contained in $\sigma \left(
U\right) \cap \sigma ^{\prime }\left( U\right) $.

By preliminary facts on affine schemes (see \cite{ega,gtm52,iitaka}) again, it is seen
that
$$\sigma |_{U_{0}}=\sigma ^{\prime }|_{U_{0}}$$ holds as
isomorphisms of schemes. Let $U_0$ run through all affine open sets of $X$, we have
$$\sigma = \sigma ^{\prime }$$ on the whole of $X$.

Hence, $t$ is an injection.

\textbf{Step 4.} Prove that $$t:Aut\left( X/Y\right) \longrightarrow Gal\left( k\left( X\right) /\phi^{\sharp}(k\left(
Y\right)) \right)$$ is a surjective homomorphism.

In fact, fixed any element
$\rho$ of the group $Gal\left( k\left( X\right)
/\phi^{\sharp}\left(k\left( Y\right)\right) \right) $.

As $$k(X)=\{[U_{f},f]:f\in \mathcal{O}_{X}(U_{f})\text{ and
}U_{f}\subseteq X\text{ is open}\},$$ we have
\begin{equation*}
\rho :[ U_{f},f] \in k\left( X\right) \longmapsto [ U_{\rho
\left( f\right) },\rho \left( f\right)] \in k\left( X\right)
\end{equation*}
according to \emph{Proposition 1.44} of \cite{iitaka},
where $U_{f}$ and $U_{\rho (f)}$ are open sets in $X$, $f$ is
contained in $\mathcal{O}_{X}(U_{f})$, and $\rho (f)$ is contained
in $ \mathcal{O}_{X}(U_{\rho (f)})$.

We will proceed in several sub-steps to prove that for
each element $$\rho\in Gal\left( k\left( X\right)
/\phi^{\sharp}\left(k\left( Y\right)\right) \right) $$ there is a
unique element $$\lambda\in {Aut}(X/Y)$$ such that $$t(\lambda)=\rho.$$

\textbf{Substep 4-a.} Fixed any affine open set $V$ of $Y$. Show that for
each affine open set $U\subseteq \phi^{-1}(V)$ there is an affine
open set $U_{\rho}$ in $X$ such that $\rho$ determines an
isomorphism between affine schemes $(U,\mathcal{O}_{X}|_{U})$ and
$(U_{\rho},\mathcal{O}_{X}|_{U_{\rho}})$.

In fact, take any affine open sets $V\subseteq Y$ and $U\subseteq \phi^{-1}(V)$. We have
\begin{equation*}
A\triangleq \mathcal{O}_{X}(U) =\{\left( U_{f},f\right):[ U_{f},f] \in k\left( X\right)
,U_{f}\supseteq U\}.
\end{equation*}
Put
\begin{equation*}
B\triangleq\{\left( U_{\rho \left( f\right) },\rho \left( f\right) \right) :[U_{\rho \left( f\right) },\rho \left( f\right)]\in
k\left( X\right),\left( U_{f},f\right) \in A
\}.
\end{equation*}

Then there exists an affine open set $W\subseteq \phi^{-1}(V)$  such that $$B=\mathcal{O}_{X}(W) $$ from the assumption that $X$ is locally complete over $Y$. Write $U_{\rho}\triangleq W$.

Therefore, by $\rho$ we have a unique isomorphism
\begin{equation*}
\lambda_{U}=\left(\lambda_{U}, \lambda_{U}^{\sharp} \right): (U, \mathcal{O}%
_{X}|_{U}) \rightarrow (U_{\rho}, \mathcal{O}_{X}|_{U_{\rho}})
\end{equation*}
of the affine open subscheme in $X$ such that
\begin{equation*}
\rho |_{\mathcal{O}_{X}(U)}=\lambda_{U}^{\sharp -1}: \mathcal{O}_{X}(U)
\rightarrow \mathcal{O}_{X}(U_{\rho}).
\end{equation*}

\textbf{Substep 4-b.} Take any affine open sets $V\subseteq Y$ and
$U,U^{\prime}\subseteq\phi^{-1}(V)$. Show that
\begin{equation*}
\lambda_{U}|_{U\cap U^{\prime}}=\lambda_{U^{\prime}}|_{U\cap U^{\prime}}
\end{equation*}
holds as morphisms of schemes.

In fact, by the above construction in $\emph{Substep 4-a}$, for each $\lambda_{U}$ it is seen
that $\lambda_{U}^{\sharp}$ and $\lambda^{\sharp}_{U^{\prime}}$
coincide restricted to the intersection $U\cap U^{\prime}$ as homomorphisms of rings since we have
\begin{equation*}
\rho |_{\mathcal{O}_{X}(U\cap U^{\prime})}=\lambda_{U}|_{U\cap U^{\prime}}^{\sharp -1}:
\mathcal{O}_{X}(U\cap U^{\prime}) \rightarrow \mathcal{O}_{X}({(U\cap U^{\prime})}_{\rho});
\end{equation*}
\begin{equation*}
\rho |_{\mathcal{O}_{X}(U\cap U^{\prime})}=\lambda_{U^{\prime}}|_{U\cap U^{\prime}}^{\sharp -1}:
\mathcal{O}_{X}(U\cap U^{\prime}) \rightarrow \mathcal{O}_{X}({(U\cap U^{\prime})}_{\rho}).
\end{equation*}

On the other hand, for any  point $x\in U\cap U^{\prime}$, we must have
$$\lambda_{U}(x)=\lambda_{U^{\prime}}(x).$$

Otherwise, if
$\lambda_{U}(x)\not=\lambda_{U^{\prime}}(x)$, will have an affine
open subset $X_0$ of $X$ that contains either of the points
$\lambda_{U}(x)$ and $\lambda_{U^{\prime}}(x)$ but does not contain
the other since the underlying space of $X$ is a Kolmogorov space (see \cite{ega,gtm52,iitaka}).
Assume $\lambda_{U}(x)\in X_{0}$ and $\lambda_{U^{\prime}}(x)\not
\in X_{0}$. We choose an affine open subset $U_{0}$ of $X$ such that
$x\in U_{0}\subseteq U\cap U^{\prime}$ and
$\lambda_{U}(U_{0})\subseteq X_{0}$ since we have
$$\lambda_{U}(U\cap U^{\prime})={(U\cap U^{\prime})}_{\rho}\subseteq
U_{\rho};$$
$$\lambda_{U^{\prime}}(U\cap U^{\prime})={(U\cap
U^{\prime})}_{\rho}\subseteq U^{\prime}_{\rho}.$$ However, for each $\lambda_{U}$, we have
$$\lambda_{U}(U_{0})=(U_{0})_{\rho}=\lambda_{U^{\prime}}(U_0);$$
then $$\lambda_{U^{\prime}}(x)\in (U_{0})_{\rho}\subseteq X_{0},$$
where there will be a contradiction. Hence, $\lambda_{U}$ and
$\lambda_{U^{\prime}}$ coincide on $U\cap U^{\prime}$ as mappings of topological
spaces.

\textbf{Substep 4-c.} By gluing the $\lambda_{U}$ along all such affine open
subsets $U$, we have a homeomorphism $\lambda$ of $X$ onto $X$ as a
topological space given by
\begin{equation*}
\lambda: x\in X \mapsto \lambda_{U}(x)\in X
\end{equation*}
where $x$ belongs to $U$ and $U$ is an affine open subset of $X$. That
is, $ \lambda|_{U}=\lambda_{U}. $

By $\emph{Substep 4-b}$ it is seen that
$\lambda $ is well-defined. Then we have an
automorphism, namely $\lambda$, of the scheme $(X,\mathcal{O}_{X})$ (see \cite{ega,gtm52,iitaka}).

\textbf{Substep 4-d.} Show that $\lambda$ is contained in $Aut\left( X/Y\right)$ satisfying
$$ t\left(\lambda\right)=\rho.$$

In deed, as $\rho$ is an isomorphism
of $k(X)$ over $\phi^{\sharp}\left(k(Y)\right)$, the
isomorphism $\lambda_{U}$ is over $Y$ by $\phi$ for any affine open
subset $U$ of $X$; then $\lambda$ is an automorphism of $X$ over $Y$
by $\phi$ such that $t\left(\lambda\right)=\rho$ holds.

This proves that  for each $\rho \in Gal\left( k\left(
X\right) /\phi^{\sharp}\left(k\left( Y\right)\right) \right) $,
there exists $\lambda\in Aut\left( X/Y\right) $
such that $$t(\lambda)=\rho.$$
Hence, ${t}$ is surjective.

By \emph{Steps 1-4} above, it is seen that $$t:Aut\left( X/Y\right) \longrightarrow Gal\left( k\left( X\right) /\phi^{\sharp}(k\left(
Y\right)) \right)$$ is an isomorphism between groups.
This completes the proof.
\end{proof}

\subsection{Proof of Theorem 2.3}

Now we give the proof of \emph{Theorem 2.3}.

\begin{proof}
(\textbf{Proof of Theorem 2.3.})
It is immediately from \emph{Definition 2.1} and \emph{Lemma 3.1}.
\end{proof}

\section{Proof of Theorem 2.6}

Let $X$ be an arithmetic variety and let $L$ be any extension of the function field $k(X)$.
In this section we will give a universal construction for an arithmetic variety $X_{L}$, a general cover  over $X$, such that $L$ is equal the function field $k(X_{L})$ (see \emph{Lemma 4.5}).  Then by the lemma we will obtain a proof of  \emph{Theorem 2.6}.

\subsection{Gluing schemes}

Let $X$ be an irreducible topological space and let $\{U_{\alpha}\}_{\alpha\in \Gamma}$ be a family of open sets in $X$ such that $X\subseteq {\bigcup_{\alpha\in \Gamma}} U_{\alpha}$. Suppose that for every $U_{\alpha}$, there is an integral domain $A_{\alpha}$ and a homeomorphism $\phi_{\alpha}$ of $U_{\alpha}$ onto the underlying space of the affine scheme $Spec A_{\alpha}$. Denote by $(U_{\alpha}, \mathcal{F}_{\alpha})$ the schemes induced from $\phi_{\alpha}$ for each $\alpha\in \Gamma$. For every $\alpha,\beta,\gamma\in \Gamma$, put
$$U_{\alpha \beta}=U_{\alpha}\bigcap U_{\beta};$$ $$U_{\alpha \beta\gamma}=U_{\alpha}\bigcap U_{\beta}\bigcap U_{\gamma}.$$

\begin{lemma}
Suppose that for every $\alpha,\beta \in \Gamma$, there is an isomorphism $$\phi_{\alpha\beta}:(U_{\alpha \beta}, \mathcal{F}_{\alpha}\mid _{U_{\alpha \beta}})\cong (U_{\alpha \beta}, \mathcal{F}_{\beta}\mid _{U_{\alpha \beta}})$$ between schemes satisfying the \textbf{cocycle condition}:
$$\phi_{\alpha\alpha}=id:(U_{\alpha}, \mathcal{F}_{\alpha})\to (U_{\alpha}, \mathcal{F}_{\alpha});$$
$$\phi_{\beta\gamma}\circ\phi_{\alpha\beta}=\phi_{\alpha\gamma}:(U_{\alpha \beta\gamma}, \mathcal{F}_{\alpha}\mid _{U_{\alpha \beta\gamma}})\to (U_{\alpha \beta\gamma}, \mathcal{F}_{\gamma}\mid _{U_{\alpha \beta\gamma}})$$
for every $\alpha,\beta,\gamma\in \Gamma$.

Then there exists a sheaf $\mathcal{F}$ on the space $X$ and a family of scheme isomorphisms $\psi_{\alpha}:(U_{\alpha},\mathcal{F}\mid_{U_{\alpha}})\to (U_{\alpha}, \mathcal{F}_{\alpha})$ such that $\phi_{\alpha\beta}=\psi_{\beta}\circ\psi_{\alpha}^{-1}$ for any $\alpha,\beta \in \Gamma$; moreover, the scheme $(X,\mathcal{F})$ is uniquely determined upon isomorphisms.

In particular, let $\phi_{\alpha}=id_{U_{\alpha}}$ for each $\alpha\in \Gamma$. Then in addition, the sheaf $\mathcal{F}$ can be such that  $\mathcal{F}\mid_{U_{\alpha}}=\mathcal{F}_{\alpha}$ for any $\alpha\in \Gamma$.
\end{lemma}

\begin{proof}
For the existence and the uniqueness of the sheaf $\mathcal{F}$, it is trivial from \cite{grauert,ega,gtm52,iitaka,serre}.

Now assume $\phi_{\alpha}=id_{U_{\alpha}}$ for each $\alpha\in \Gamma$. Let $\mathcal{F}$ be a sheaf on $X$ determined under the cocycle condition such that
$$\psi_{\alpha}:(U_{\alpha},\mathcal{F}\mid_{U_{\alpha}})\cong (U_{\alpha}, \mathcal{F}_{\alpha})$$ for every $\alpha\in \Gamma$.

In the following prove that there is a sheaf $\mathcal{F}^{\prime}$ on $X$ determined under the cocycle condition such that  in addition we have  $$\mathcal{F}^{\prime}\mid_{U_{\alpha}}= \mathcal{F}^{\prime}_{\alpha}= \mathcal{F}_{\alpha}$$ for any $\alpha\in \Gamma$.

In fact, fixed any  point $x\in X$. Put
$$\mathcal{U}_{x}=\text{ the family of open sets in} X\text{ that contains the point} x;$$
$$\mathcal{U}^{\Gamma}_{x}=\text{ the family of affine open sets in } U_{\alpha}\text{ with }\alpha\in \Gamma\text{ that contains the point }x.$$

By set-inclusion, $\mathcal{U}^{\Gamma}_{x}$ and $\mathcal{U}_{x}$ are two directed sets. In particular, $\mathcal{U}^{\Gamma}_{x}$ is  cofinal in $\mathcal{U}_{x}$ since $\{U_{\alpha}\}_{\alpha\in \Gamma}$ is an open covering of $X$. Then we have the direct limits
$$\mathcal{F}^{\prime}_{x}\triangleq\lim_{U_{\alpha}\supseteq U\in\mathcal{U}^{\Gamma}_{x}}\mathcal{F}_{\alpha}(U);$$
$$\mathcal{F}_{x}=\lim_{U\in\mathcal{U}_{x}}\mathcal{F}(U)=\lim_{U_{\alpha}\supseteq U\in\mathcal{U}^{\Gamma}_{x}}\mathcal{F}(U).$$ It is clear that the canonical maps $\mathcal{F}_{\alpha}(U)\to \mathcal{F}^{\prime}_{x}$ and $\mathcal{F}_{\alpha}(U)\to \mathcal{F}_{x}$ are both injective. Hence, we can identify them with their images, respectively.

Via $\{\psi_{\alpha}\}_{\alpha\in\Gamma}$, there is an induced isomorphism $$\psi_{x}:\mathcal{F}_{x}\cong \mathcal{F}^{\prime}_{x}$$ between the direct limits according to preliminary facts on direct limits of direct systems of rings.

For every open set $U$ in $X$, define $$\mathcal{F}^{\prime}(U)=\psi_{x}(\mathcal{F}(U)),$$ i.e., $\mathcal{F}^{\prime}(U)$ is the natural image of $\mathcal{F}(U)$ by the map $\psi_{x}$, where $x$ is a point contained in $U$. It is independent of the choice of the point $x$.

Then $\mathcal{F}^{\prime}$ is a sheaf on $X$ satisfying the desired property.
\end{proof}

\subsection{Essentially affine schemes}

Here we give a type of schemes that behave almost like affine schemes.

\begin{definition}
(\cite{0909,0910})
An arithmetic variety $X$ is said to be \textbf{essentially affine} in a field $\Omega $ if every affine open set $U$ of $X$
satisfies the two properties: $$\mathcal{O}_{X}(U)\subseteq \Omega ;$$ $$U=Spec(\mathcal{O}
_{X}(U)).$$
\end{definition}

Here we give such a fact that every arithmetic variety $X$ is isomorphic to an
arithmetic variety $Z$ that is essentially affine in the field $k(X)$.

\begin{lemma}
\emph{(\cite{0909,0910})}
For any arithmetic variety $X$, there is an arithmetic variety $Z$ satisfying the
properties:
\begin{itemize}
\item $k\left( X\right) =k\left( Z\right);$

\item $X\cong Z$ are isomorphic schemes;

\item $Z$ is essentially affine in the field $k(Z)$.
\end{itemize}
\end{lemma}

Such an arithmetic variety $Z$ is called an \textbf{essentially affine realization} of $X$.

\begin{proof}
Denote by $\{W_{\alpha}\}_{\alpha\in \Gamma}$ the collection of all affine open sets in the scheme $X$ and denote by $$i_{\alpha}:\mathcal{O}_{X}(W_{\alpha})\to A_{\alpha}=i_{\alpha}(\mathcal{O}_{X}(W_{\alpha}))\subseteq k(X)$$  the canonical embedding for every $\alpha\in \Gamma$.

Let $$\widetilde{X}=\coprod_{\alpha\in \Gamma}W_{\alpha}$$ and $$\widetilde{Z}=\coprod_{\alpha\in \Gamma}U_{\alpha}$$
be the disjoint union of sets, respectively. Here, $U_{\alpha}$ denotes the underlying space of the affine scheme $Spec A_{\alpha}$ for any $\alpha\in \Gamma$.

For any points $x_{\alpha}\in W_{\alpha},x_{\beta}\in W_{\beta}$, we say $$x_{\alpha}\sim x_{\beta}$$ if and only if $$x_{\alpha}=x_{\beta}$$ are the same point in $X$.

Likewise, for any points $z_{\alpha}\in U_{\alpha},z_{\beta}\in U_{\beta}$, we say $$z_{\alpha}\sim z_{\beta}$$ if and only if $$\rho_{\alpha}(z_{\alpha})\sim\rho_{\beta}(z_{\beta})$$ holds in $\widetilde{X}$, where $\rho_{\alpha}$ denotes the homeomorphism of $U_{\alpha}$ onto $W_{\alpha}$ that is induced from the homomorphism $i_{\alpha}$ for any $\alpha\in \Gamma$.

Then we have the quotient sets $$X=\widetilde{X}/\sim;$$ $$Z=\widetilde{Z}/\sim.$$
It is seen that there is a bijection $\rho:Z\cong X$ between sets since $\widetilde{X}\cong\widetilde{Z}$ are bijective sets from all the $\rho_{\alpha}$.

Hence, by $\rho$, $Z$ is a topological space homeomorphic to $X$. Here, for any subset $U\subset Z$, we say $U$ is open in $Z$ if and only if $\rho(U)$ is an open set in $X$.

Fixed any $\alpha,\beta\in\Gamma$. Define $$U_{\alpha\beta}\triangleq U_{\alpha}\bigcap U_{\beta};$$
$$(U_{\alpha},\mathcal{F}_{\alpha})\triangleq({Spec}{A_{\alpha}},\widetilde{A_{\alpha}}).$$
We prove $$\mathcal{F}_{\alpha}\mid_{U_{\alpha\beta}}=\mathcal{F}_{\beta}\mid_{U_{\alpha\beta}}.$$

In deed, take any affine open subset
$$ U_{\gamma_{0}}={Spec}{A_{\gamma_{0}}}$$ contained in the set $ U_{\alpha\beta}$.
On the affine scheme $(U_{\alpha},\mathcal{F}_{\alpha})$, we have
$$\mathcal{F}_{\alpha}\mid_{U_{\gamma_{0}}}=\mathcal{F}_{\gamma_{0}}=\widetilde{A_{\gamma_{0}}};$$
on the affine scheme $(U_{\beta},\mathcal{F}_{\beta})$, we have
$$\mathcal{F}_{\beta}\mid_{U_{\gamma_{0}}}=\mathcal{F}_{\gamma_{0}}=\widetilde{A_{\gamma_{0}}}.$$
Then $$\mathcal{F}_{\alpha}\mid_{U_{\gamma_{0}}}=\mathcal{F}_{\beta}\mid_{U_{\gamma_{0}}}$$
holds for any such $U_{\gamma_{0}}$. Hence,
we must have
$$\mathcal{F}_{\alpha}\mid_{U_{\alpha\beta}}=\mathcal{F}_{\beta}\mid_{U_{\alpha\beta}}$$ by taking an open covering $$U_{\alpha\beta}=\bigcup_{\gamma_{0}}U_{\gamma_{0}}$$
since $\mathcal{F}_{\alpha}\mid_{U_{\alpha\beta}}$ and $\mathcal{F}_{\beta}\mid_{U_{\alpha\beta}}$ are two sheaves on the same space $U_{\alpha\beta}$, respectively.

Now apply \emph{Lemma 4.1} to this case here: $$U_{\alpha}=\rho^{-1}(W_{\alpha});$$ $$ \phi_{\alpha}=id_{U_{\alpha}};$$
$$\phi_{\alpha\beta}=id:(U_{\alpha \beta}, \mathcal{F}_{\alpha}\mid _{U_{\alpha \beta}})= (U_{\alpha \beta}, \mathcal{F}_{\beta}\mid _{U_{\alpha \beta}}).$$
The cocycle condition is automatically satisfied.

Hence, there is a sheaf $\mathcal{F}$ on $Z$ such that  $\mathcal{F}\mid_{U_{\alpha}}=\mathcal{F}_{\alpha}$ holds for any $\alpha\in \Gamma$. We have a scheme $(Z,\mathcal{F})$ satisfying the desired properties, where the isomorphism $$(Z,\mathcal{F})\cong (X,\mathcal{O}_{X})$$ is induced from $\rho$ just by considering the corresponding stalks at each point.
\end{proof}

\subsection{Construction for a general cover}

Let's consider such a preliminary fact. Let $L$ be an extension of a field $K$ and let $\Delta\subseteq L\setminus K$ be a subset.

\begin{remark}
Put
$\Delta _{L/K}\triangleq \{\sigma \left(x\right) \in L:x\in \Delta,\sigma \in Gal\left(
L/K\right)\}.$ Then we have $\delta(\Delta _{L/K})=\Delta _{L/K}$ for each automorphism $\delta\in Gal(L/K)$.
\end{remark}

Now we give a universal construction of a general cover over an arithmetic variety such that the function field of the cover can be equal to any prescribed filed.

\begin{lemma}
Let $Y$ be an arithmetic variety. Take any extension $L$ of the function field $k(Y)$ (not necessarily algebraic). Then there is an arithmetic variety $X_{L}$ and a  morphism $\phi_{L}:X_{L}\to Y$ such that
\begin{itemize}
\item $k(X_{L})=L$;

\item $\phi_{L}$ is affine and surjective;

\item $Aut(X_{L}/Y)\cong Gal(L/k(Y))$ are naturally isomorphic groups.
\end{itemize}
\end{lemma}

\begin{proof}
(The approach here is based on a
trick originally in \cite{0909} and a refinement in \cite{0910}.)

From \emph{Lemma 4.3}, suppose that $Y$ is essentially affine in the function field $k(Y)$ without loss
of generality.

Denote by $\{V_{\alpha}\}_{\alpha\in \Gamma}$ the collection of all affine open sets $V_{\alpha}$ in the scheme $Y$. Let $B_{\alpha}=\mathcal{O }_{Y}\left( V_{\alpha}\right)$ for each $\alpha\in \Gamma$. We have $V_{\alpha}={Spec}{B_{\alpha}}$ and $B_{\alpha}\subseteq k(Y)\subseteq L$.

Suppose that
\begin{equation*}
\Delta\subseteq L\setminus k(Y)
\end{equation*}
is a set of generators of the field $L$ over the function field $k(Y)$. Set

$$\Delta _{L/k(X)}\triangleq \{\sigma \left(x\right) \in L:x\in \Delta,\sigma \in Gal\left(
L/k(Y)\right)\}.$$

In the following we will proceed in several
steps to construct an arithmetic variety $X=X_{L}$ and a morphism $\phi=\phi_{L}:X_{L}\to Y$.

{\textbf{Step 1.}} For every $\alpha\in \Gamma$, define
\begin{equation*}
A_{\alpha}\triangleq B_{\alpha}\left[ \Delta _{L/k(Y)}\right],
\end{equation*}
i.e., $A_{\alpha}$ is a  subring of $L$ generated over $B_{\alpha}$ by the set $\Delta _{L/k(Y)}$.

We have $Fr\left( A_{\alpha}\right) =L $. As $A_{\alpha}$ is a subring of $L$, by \emph{Remark 4.7} it is seen that the restriction $\delta\mid_{A_{\alpha}}:A_{\alpha}\to A_{\alpha}$ is an isomorphism for every $\delta\in Gal(L/k(Y))$.

Set
\begin{equation*}
i_{\alpha}:B_{\alpha} \hookrightarrow  A_{\alpha}
\end{equation*}
to be the inclusion.

{\textbf{Step 2.}} Define the disjoint unions
\begin{equation*}
\widetilde{X}=\coprod\limits_{\alpha\in\Gamma}Spec\left( A_{\alpha}\right) ;
\end{equation*}
\begin{equation*}
\widetilde{Y}=\coprod\limits_{\alpha\in\Gamma}Spec\left( B_{\alpha}\right) .
\end{equation*}

Then $\widetilde{X} $ is a topological space, where the topology $\tau _{\widetilde{X} }$
on $\widetilde{X} $ is naturally determined by the Zariski topologies on all $%
Spec\left( A_{\alpha}\right) .$

Denote by
\begin{equation*}
\pi _{1}:\widetilde{X} \rightarrow \widetilde{Y}
\end{equation*}
the natural map induced from the inclusions $i_{\alpha}$. It is seen that $\pi _{1}$ is a surjection just by considering the inclusion $i_{\alpha}:B_{\alpha}\rightarrow A_{\alpha}$ for every $\alpha\in\Gamma$ (see \cite{ega,iitaka}).

Let
\begin{equation*}
\pi _{0}:\widetilde{Y} \rightarrow Y=\widetilde{Y}/\sim
\end{equation*}
be the projection, where for any points $y_{\alpha}\in Spec B_{\alpha},y_{\beta}\in Spec B_{\beta}$, we say $$y_{\alpha}\sim y_{\beta}$$ if and only if $$y_{\alpha}= y_{\beta}$$ holds in $Y$.

Set
$$\pi_{Y}\triangleq\pi _{0}\circ \pi _{1}: \widetilde{X} \to Y$$ to be the projection.

{\textbf{Step 3.}} For any points $x_{\alpha}\in Spec A_{\alpha},x_{\beta}\in Spec A_{\beta}$, we say $$x_{\alpha}\sim x_{\beta}$$ if and only if $$y_{\alpha}=\pi_{1}(x_{\alpha})= y_{\beta}=\pi_{1}(x_{\beta})\in V_{\alpha\beta}$$ holds in $\widetilde{Y}$, i.e.,
if and only if $$\pi _{Y}(x_{\alpha})=\pi _{Y}(x_{\beta})\in V_{\alpha\beta}$$ holds\footnote{Note that here the equivalence of affine $L$-points, used in our earlier preprints, is automatically satisfied.} as points in $Y$. Here, $V_{\alpha\beta}\triangleq V_{\alpha}\bigcap V_{\beta}$.

Let
\begin{equation*}
X=\widetilde{X} /\sim
\end{equation*}
and let
\begin{equation*}
\pi _{X}:\widetilde{X} \rightarrow X
\end{equation*}
be the projection.

It is seen that $X$ is a topological space as a quotient of $\widetilde{X} .$

{\textbf{Step 4.}} Define a map
\begin{equation*}
\phi:X\rightarrow Y
\end{equation*}
by
\begin{equation*}
\pi _{X}\left( \widetilde{x}\right) \longmapsto \pi _{Y}\left( \widetilde{x}\right)
\end{equation*}
for each $\widetilde{x}\in \widetilde{X} $.

By \emph{Step 3} above it is seen that $\phi$ is independent of the choice of a representative $\widetilde{x}$ and hence $\phi$ is well-defined.

It is clear that $\phi$ is a surjection. For every $\alpha\in \Gamma$, put
$$U_{\alpha}=\phi^{-1}(V_{\alpha}).$$
Then we have $$U_{\alpha}={Spec}{A_{\alpha}};$$
$$U_{\alpha\beta}\triangleq U_{\alpha}\bigcap U_{\beta}=\phi^{-1}(V_{\alpha\beta})$$
for every $\alpha,\beta\in \Gamma$.

{\textbf{Step 5.}} For every $\alpha\in \Gamma$, we define
$$(U_{\alpha},\mathcal{F}_{\alpha})\triangleq({Spec}{A_{\alpha}},\widetilde{A_{\alpha}}).$$

Fixed any $\alpha,\beta\in\Gamma$. We prove $$\mathcal{F}_{\alpha}\mid_{U_{\alpha\beta}}=\mathcal{F}_{\beta}\mid_{U_{\alpha\beta}}.$$

In deed, take any point $y_{0}\in V_{\alpha\beta}=\phi(U_{\alpha\beta})$. There is a $\gamma_{0}\in\Gamma$ such that $$y_{0}\in V_{\gamma_{0}}={Spec}{B_{\gamma_{0}}}\subseteq V_{\alpha\beta}\subseteq V_{\alpha}\bigcap V_{\beta}.$$
Then for any $x_{0}\in U_{\alpha\beta}$, we have
$$x_{0}\in U_{\gamma_{0}}={Spec}{A_{\gamma_{0}}}\subseteq U_{\alpha\beta}\subseteq U_{\alpha}\bigcap U_{\beta}$$
according to \emph{Step 4} above.

At the same time, as $U_{\gamma_{0}}=\phi^{-1}(V_{\gamma_{0}})$, on the affine scheme $(U_{\alpha},\mathcal{F}_{\alpha})$, we have
$$\mathcal{F}_{\alpha}\mid_{U_{\gamma_{0}}}=\mathcal{F}_{\gamma_{0}}=\widetilde{A_{\gamma_{0}}};$$
on the affine scheme $(U_{\beta},\mathcal{F}_{\beta})$, we have
$$\mathcal{F}_{\beta}\mid_{U_{\gamma_{0}}}=\mathcal{F}_{\gamma_{0}}=\widetilde{A_{\gamma_{0}}}.$$
Hence, $$\mathcal{F}_{\alpha}\mid_{U_{\gamma_{0}}}=\mathcal{F}_{\beta}\mid_{U_{\gamma_{0}}}$$
holds for any such $U_{\gamma_{0}}$.

From the open covering $U_{\alpha\beta}=\bigcup_{\gamma_{0}}U_{\gamma_{0}}$, we must have
$$\mathcal{F}_{\alpha}\mid_{U_{\alpha\beta}}=\mathcal{F}_{\beta}\mid_{U_{\alpha\beta}}$$
since $\mathcal{F}_{\alpha}\mid_{U_{\alpha\beta}}$ and $\mathcal{F}_{\beta}\mid_{U_{\alpha\beta}}$ are two sheaves on the same space $U_{\alpha\beta}$, respectively.

{\textbf{Step 6.}}
Now apply \emph{Lemma 4.1} to this case here:
$$U_{\alpha}={Spec}{A_{\alpha}};$$
$$\phi_{\alpha}\triangleq id_{U_{\alpha}}:U_{\alpha}\to {Spec}{A_{\alpha}}.$$
$$\phi_{\alpha\beta}=id:(U_{\alpha \beta}, \mathcal{F}_{\alpha}\mid _{U_{\alpha \beta}})= (U_{\alpha \beta}, \mathcal{F}_{\beta}\mid _{U_{\alpha \beta}}).$$
The cocycle condition is automatically satisfied.

Hence, there is a sheaf $\mathcal{F}$ on $X$ such that  $\mathcal{F}\mid_{U_{\alpha}}=\mathcal{F}_{\alpha}$ holds for any $\alpha\in \Gamma$. We have a scheme $(X,\mathcal{F})$ that is essentially affine in the field $L=k(X)$.

{\textbf{Step 7.}} Prove that the map $\phi:X\to Y$ induces a morphism, namely $\phi=(\phi,\phi^{\sharp})$, from the scheme $(X,\mathcal{F})$ onto the scheme $(Y,\mathcal{O}_{Y})$.

In fact, for every affine open set $V_{\alpha}$ in $Y$, the restriction $$\phi\mid_{U_{\alpha}}:(U_{\alpha},\mathcal{F}_{\alpha})\to (V_{\alpha},\mathcal{O}_{Y}\mid_{V_{\alpha}})$$ is a morphism of affine schemes which is  induced from the embedding $i_{\alpha}:B_{\alpha}\hookrightarrow A_{\alpha}$ of rings. That is, $$\phi^{\sharp}\mid_{V_{\alpha}}=i_{\alpha}:\mathcal{O}_{Y}\mid_{V_{\alpha}}(V_{\alpha})=B_{\alpha}\hookrightarrow A_{\alpha}=\mathcal{F}_{\alpha}(U_{\alpha}).$$

Fixed any open subset $V$ of $Y$. Take any open covering $V=\bigcup_{i}V_{\alpha_{i}}$, where each $V_{\alpha_{i}}$ is an affine open set in $Y$. For every $t\in \mathcal{O}_{Y}(V)$, we have the restriction $$t_{\alpha_{i}}\triangleq t\mid_{V_{\alpha_{i}}};$$
then define the image $\phi^{\sharp}(t)$ to the element in $\mathcal{F}(U)$ determined by the elements $$\phi^{\sharp}\mid_{V_{\alpha_{i}}}(t_{\alpha_{i}})\in \mathcal{F}(U_{\alpha_{i}})$$
where $$U_{\alpha_{i}}=\phi^{-1}(V_{\alpha_{i}});$$
$$U=\phi^{-1}(V)=\bigcup_{i}U_{\alpha_{i}}.$$ Here, all $U_{\alpha_{i}}$ are affine open sets in $X$.

Hence, we have a morphism $\phi=(\phi,\phi^{\sharp}):(X,\mathcal{F})\to (Y,\mathcal{O}_{Y})$ of schemes. It is clear that $\phi$ is affine and surjective.

{\textbf{Step 8.}} Prove that $(X,\mathcal{F})$ is locally complete over  $ (Y,\mathcal{O}_{Y})$ by the morphism $\phi=(\phi,\phi^{\sharp})$.

In fact, given any automorphism $\rho\in Gal(L/k(Y))$ and any affine open sets $V_{\alpha}\subseteq Y$ and $U_{\alpha}=\phi^{-1}(V_{\alpha})$. Then we must have an isomorphism $$\widetilde{\rho}:\mathcal{F}(U_{\alpha})=A_{\alpha}\to \mathcal{F}(U_{\alpha})=A_{\alpha}$$ of rings  that induces an isomorphism $$\rho:L=Fr(A_{\alpha})\to L=Fr(A_{\alpha})$$ of fields just by taking the restriction $$\widetilde{\rho}=\rho\mid_{A_{\alpha}}:A_{\alpha} \to A_{\alpha}$$ since $\rho(A_{\alpha})=A_{\alpha}$ holds from \emph{Lemma 4.4}.

Hence, we have $$Aut(X/Y)\cong Gal(L/k(Y))$$ from \emph{Lemma 3.1}. This completes the proof.
\end{proof}

\subsection{Proof of Theorem 2.6}

Now we give the proof of \emph{Theorem 2.6}.

\begin{proof}
(\textbf{Proof of Theorem 2.6.})
It is immediately from \emph{Definition 2.5} and \emph{Lemma 4.5}.
\end{proof}

\bigskip

\textbf{Acknowledgment.} The author would like to express his
sincere gratitude to Professor Li Banghe for his advice and instructions on
algebraic geometry and topology.


\newpage

\begin{thebibliography}{99}

\bibitem{0907} An, F-W. Automorphism groups of quasi-galois closed arithmetic
schemes. eprint arXiv:0907.0842.

\bibitem{0909} An, F-W. On the existence of geometric models for function
fields in several variables. eprint arXiv:0909.1993.

\bibitem{0910} An, F-W. On the \'{e}tale fundamental groups of arithmetic schemes. eprint arXiv:0910.4646.

\bibitem{1001} An, F-W. On the quasi-galois cover of an integral scheme. preprint, 2010.

\bibitem{1004} An, F-W. On the transcendental Galois extensions. eprint arXiv:1004.5036.

\bibitem{1006} An, F-W. On the unramified extension of an arithmetic function field in several variables. eprint arXiv:1006.5143.

\bibitem{grauert} Grauert, H; Remmert, R. Theory of Stein Spaces. Springer, Berlin, 2004.

\bibitem{ega} Grothendieck, A; Dieudonn\'{e}, J. A. \'{E}l\'{e}ments de
G\'{e}oem\'{e}trie Alg\'{e}brique I. Springer, New York, 1971.

\bibitem{sga1} Grothendieck, A; Raynaud, M. Rev$\hat{e}$tements
$\acute{E}$tales et Groupe Fondamental (SGA1). Springer, New York, 1971.

\bibitem{gtm52} Hartshorne, R. Algebraic Geometry. Springer, New York, 1977.

\bibitem{iitaka} Iitaka, S. Algebraic Geometry. Springer, New York, 1982.

\bibitem{lang} Lang, S. Unramified Class Field Theory Over Function
Fields in Several Variables. Annals of Math, 2nd Ser., Vol 64, No. 2
(1956), 285-325.

\bibitem{neu} Neukirch, J. Algebraic Number Theory. Springer, Berlin, 1999.

\bibitem{serre} Serre, J-P. Faisceaux Algebriques Coherents. Annals of Math., Vol 61, No. 2 (1955), 197-278.

\bibitem{s-v1} Suslin, A; Voevodsky, V. Singular homology of abstract
algebraic varieties. Invent. Math. 123 (1996), 61-94.

\bibitem{s-v2} Suslin, A; Voevodsky, V. Relative cycles and Chow sheaves, in
\emph{Cycles, Transfers, and Motivic Homology Theories}, ed. by Voevodsky, V;
Suslin, A; Friedlander, E. Annals of Math Studies, Vol 143. Princeton
University Press, Princeton, NJ, 2000.
\end{thebibliography}
\end{document}